%% file: main.tex
\documentclass[11pt,a4paper,final]{article}

\input{preamble}

\usepackage{tikz}
\usetikzlibrary{arrows.meta,calc,decorations.pathreplacing,patterns}

\newcommand{\dd}{\,d}
\newcommand{\Om}{\Omega}
\newcommand{\Omst}{\Omega_*}
\newcommand{\one}{\mathbf 1}

\title{An improvement on B-spline basis condition number}
\author{Yimin Zhong~\thanks{Department of Mathematics and Statistics, Auburn University, Auburn, AL 36849; yimin.zhong@auburn.edu}}

\date{}

\begin{document}
\maketitle

\begin{abstract}
In this short note, we improve the upper bound on the condition number of the univariate B-spline basis of order $k$ from $k 2^k$ to $O(\sqrt{k}\log k\,2^k)$ for all $k\ge 2$. 
\end{abstract}


\section{Introduction}

The stability of the B-spline basis is a classical problem in approximation theory.  Carl de Boor conjectured that the relevant B-spline basis condition number grows essentially like $2^k$ as the spline order $k$ increases; see \cite{deBoor1976Survey,deBoor1976Local,deBoor1978,deBoor1990} for discussions.  Earlier work by Lyche~\cite{Lyche1978} and by Scherer and Shadrin~\cite{SchererShadrin1996} established increasingly sharp exponential estimates. Later on, Scherer and Shadrin  proved an exciting upper bound
$$
        \kappa_{k,p}<k2^k,
$$
for all $k$ and all $1\le p\le\infty$ \cite{SchererShadrin1999}, which leaves only a linear factor away from the final conjecture.  Their proof uses Lee's identity \cite{Lee1996}, which yields a local interpolation formula, and finally estimates a binomial summation term-by-term with uniform bounds. 

The goal of this work is to improve Scherer and Shadrin's upper bound by estimating the same quantity. Instead of summing the absolute bounds with binomial weights, we show that the binomial sum has additional cancellation, leading to an improved estimate for the condition number 
$$\kappa_{k, p} < 3 \cdot 2^k k^{1/2}(3 + \log k) ,$$
which pushes the upper bound to almost the optimal rate mentioned in~\cite{SchererShadrin1999} for this route. The key observation is that the generating polynomial of the coefficients in Lee's identity admits an exact Cauchy integral representation, which can be transformed into a contour integral along the rays passing through the origin. On each ray, the integrand can be factored into unimodular factors, which allows us to exploit an extremely unbalanced setting in which all positive (or negative) interior knots are close to zero. Such exploitation eventually removes $O(\sqrt{k})$ from the condition number up to a logarithmic factor.

The subsequent sections of this note are organized as follows. We first introduce the previous result in~\cite{SchererShadrin1999} and the main theorem in Section~\ref{sec: main thm}. In Section~\ref{sec: pre}, we introduce a few preliminary results and explain their motivations. Then, we derive the main estimates in~\ref{sec: est}. Those estimates eventually constitute the proof in Section~\ref{sec: proof}. The concluding remarks are given in Section~\ref{sec: con}.

\section{Main theorem}\label{sec: main thm}
Let $\Delta = (t_0, t_1, \ldots, t_k)$ be the knot vector that 
$$
t_0\le t_1 \le \cdots \le t_k.
$$
For simplicity, we assume the knots are distinct, and repeated knots can be viewed as a conventional limiting process. If $f$ is a function, then $[\zeta_0,\ldots, \zeta_m] f$ denotes the divided difference of $f$ at the nodes $(\zeta_0,\ldots, \zeta_m)$.

We follow the same notations in~\cite{SchererShadrin1999} and define 
\begin{equation*}
\begin{aligned}
    \omega(x) &:= \frac{1}{(k-1)!}\prod_{j=1}^{k-1} (x - t_j),\\
    N(x) &:= ([t_0, \ldots, t_{k-1}] - [t_1,\ldots, t_k]) (\cdot - x)_{+}^{k-1}, 
\end{aligned}
\end{equation*}
where $(z)_{+} = \max(0, z)$ for $z\in\bbR$. The upper bound of the condition number $\kappa_{k, p}$ has been estimated in~\cite{SchererShadrin1999} using a useful identity from~\cite{Lee1996}.
\begin{theorem}[\cite{SchererShadrin1999}]\label{thm: SS}
The condition number $\kappa_{k,p}$, $p\ge 1$, satisfies
\begin{equation}\label{eq: bi-sum}
   \kappa_{k,p} \le \sup_{x\in (t_0, t_k)} \left| k \sum_{m=1}^k \binom{k}{m} N^{(m-1)}(x) \omega^{(k-m)}(x)\right|,
\end{equation}
and 
\begin{equation}\label{eq: bnd}
    \left| N^{(m-1)}(x) \omega^{(k-m)}(x) \right| \le 1, \qquad m = 1,\ldots, k.
\end{equation}
\end{theorem}
A direct term-by-term application of the bound~\eqref{eq: bnd} in Theorem~\ref{thm: SS} immediately yields $\kappa_{k,p}\le k 2^{k}$. The main contribution of the current work is an improvement over the binomial summation~\eqref{eq: bi-sum} in Theorem~\ref{thm: SS} as follows.

\begin{theorem}\label{thm:main}
For $k\ge 2$ and any $p\ge 1$, the condition number satisfies
\begin{equation}\label{eq: new bnd}
    \kappa_{k,p} < 3 \cdot 2^k  k^{1/2} (3 + \log k).
\end{equation}
\end{theorem}
Instead of dealing with~\eqref{eq: bi-sum} term-by-term, we treat it in integral form using an auxiliary generating polynomial, reducing the problem to estimating Cauchy integrals. The leading constant $3$ in the theorem can be further optimized to approximately $\frac{4\sqrt{2}}{\pi^{3/2}}$ as $k\to\infty$ according to the proof in Section~\ref{sec: proof}, but we choose to keep the result in a much simpler form~\eqref{eq: new bnd}.

\section{Preliminaries}\label{sec: pre}
In the following, we introduce a few important preliminary steps. 
\subsection{Generating polynomial}
For a fixed point $x$, we introduce the generating polynomial
$$U(y; x) = \sum_{j=0}^{k-1}u_j(x) y^j,\qquad u_j(x) = N^{(k-1-j)}(x) \omega^{(j)}(x).$$
Then Theorem~\ref{thm: SS} implies $|u_j(x)| \le 1$ for all $j=0,\ldots, k-1$. The upper bound of $\kappa_{k,p}$ amounts to estimating the summation 
\begin{equation}\label{eq: bi sum sigma}
    \Sigma_k(x) = \sum_{j=0}^{k-1} \binom{k}{j} u_j(x).
\end{equation}
Set $\eta_j = t_j - x$, $j=0,\ldots, k$. Without loss of generality, we may assume $\eta_0 < 0 < \eta_k$. Then, we define the following auxiliary polynomials
$$
\Om(z):=\prod_{i=0}^{k}(z-\eta_i),
\qquad
\Omst(z):=\prod_{i=1}^{k-1}(z-\eta_i),  \qquad z\in \bbC.
$$
We first prove the following identity, using the same approach as in~\cite{Lee1996}. 
\begin{lemma}\label{lem: U}
$U(y; x)$ satisfies 
    \begin{equation}\label{eq:U-def}
        U(y; x):=-(\eta_k - \eta_0) \cdot [\eta_0,\ldots,\eta_k]
        \bigl(\Omst(-y z)\one_{z >0}\bigr),
\end{equation}
where the divided difference is applied to the $z$ variable.
\end{lemma}
\begin{proof}
 Let $Q_{\delta_1}$ and $Q_{\delta_2}$ be the interpolation polynomials in $z$ for $\phi_{k-1}(z; t) := \frac{1}{(k-1)!}(z - t)_{+}^{k-1}$ at nodes 
    $$\delta_1:= (\eta_0, \ldots, \eta_{k-1}),\qquad \delta_2 := (\eta_1,\ldots, \eta_k).$$
    Since $Q_{\delta_1}(z) = Q_{\delta_2}(z)$ at the nodes $(\eta_1,\ldots, \eta_{k-1})$,
    \begin{equation*}
        Q_{\delta_1}(z, t) - Q_{\delta_2}(z, t) = c(t) \Omega_{\ast}(z),
    \end{equation*}
    where the leading coefficient $c(t)$ is
    \begin{equation*}
        c(t) =  ([\eta_0,\ldots, \eta_{k-1}] - [\eta_1,\ldots, \eta_k]) \phi(\cdot; t) = (\eta_0 - \eta_k)[\eta_0 ,\ldots, \eta_k] \phi(\cdot; t).
    \end{equation*}
    Next, we compute each coefficient $u_j(x)$. By the shifting of variables, we have $$c^{(k-1-j)}(0) = \frac{1}{(k-1)!} N^{(k-1-j)}(x),\qquad \omega^{(j)}(x) = \frac{(-1)^j}{(k-1)!} \Omega_{\ast}^{(j)}(0).$$ 
    Therefore, 
    \begin{equation*}
    \begin{aligned}
        u_j(x) &= N^{(k-1 - j)}(x) \omega^{(j)}(x) = (-1)^j c^{(k-1-j)}(0) \Omega_{\ast}^{(j)}(0) \\
        &=   (\eta_0 - \eta_k)[\eta_0 ,\ldots, \eta_k] g_j(\cdot), 
    \end{aligned}
    \end{equation*}
    where $h_j(z) =  \frac{(-1)^j}{j!}\Omega_{\ast}^{(j)}(0) (z)_{+}^j$. Noticing that 
    \begin{equation*}
        \sum_{j=0}^{k-1} u_j(x) y^j = (\eta_0 - \eta_k)[\eta_0 ,\ldots, \eta_k] \left( \sum_{j=0}^{k-1} h_j(\cdot) y^j \right) = (\eta_0 - \eta_k)[\eta_0 ,\ldots, \eta_k] \Omega_{\ast}(-y(z)_{+}),
    \end{equation*}
    which is exactly~\eqref{eq:U-def} by writing $\Omega_{\ast}$ into a Taylor series.
\end{proof}
The Lemma~\ref{lem: U} also implies that we can always treat the evaluation point $x$ at the origin with the shifted knots. Hence, from now on, we just use $U(y)$ instead of $U(y;x)$.

\subsection{Cauchy integrals}
It is a standard technique to write the divided difference in terms of the Cauchy integral over the nodes.

\begin{lemma}\label{lem:contour}
Let $g$ be a polynomial such that $\deg(g)\le k-1$.  Let $\Gamma_+$ be a positively oriented contour enclosing \textit{precisely} the positive nodes among $(\eta_0,\ldots,\eta_k)$.  Then
\begin{equation}\label{eq:contour-divdiff}
        [\eta_0,\ldots,\eta_k](g(z)\one_{z>0})
        =\frac1{2\pi i}\int_{\Gamma_+}\frac{g(\zeta)}{\Om(\zeta)}\dd\zeta.
\end{equation}
Consequently,
\begin{equation}\label{eq:U-contour}
        U(y)=-\frac{\eta_k - \eta_0}{2\pi i}\int_{\Gamma_+}
        \frac{\Omst(-y\zeta)}{\Om(\zeta)}\dd\zeta.
\end{equation}
\end{lemma}

\begin{proof}
For any knot value $w=\eta_i$, we define
$$
        \chi_+(w)=\frac1{2\pi i}\int_{\Gamma_+}\frac{\dd\zeta}{\zeta-w}.
$$
Then $\chi_+(\eta_i)=1$ for $\eta_i > 0$ and $\chi_+(\eta_i)=0$ for $\eta_i < 0$.  Therefore the interpolation data of $g(x)\one_{x>0}$ agree with those of $g(x)\chi_+(x)$.  Hence
$$
 [\eta_0,\ldots,\eta_k](g(x)\one_{x>0})
 =\frac1{2\pi i}\int_{\Gamma_+}
 [\eta_0,\ldots,\eta_k]\left(\frac{g(x)}{\zeta-x}\right)\dd\zeta .
$$
For fixed $\zeta$, write
$$
        \frac{g(x)}{\zeta-x}
        =\frac{g(\zeta)}{\zeta-x}
        +\frac{g(x)-g(\zeta)}{\zeta-x}.
$$
The second term is a polynomial in $x$ of degree at most $k-2$, so its $k$-th divided difference is zero.  The first term has $k$-th divided difference $g(\zeta)/\Om(\zeta)$.  Taking $g(z)=\Omst(-yz)$ proves \eqref{eq:U-contour}.
\end{proof}

Let $y = e^{i\theta}$, $|\theta| < \pi$. Define
\begin{equation*}
    L_{\theta} := \{ r e^{ - i (\pi + \theta) / 2}\mid r\in\bbR \}.
\end{equation*}
For each $|\theta| < \pi$, $L_{\theta}$ separates the positive and negative real half-axes. 

\begin{lemma}\label{lem: rot conj}
    For each $\zeta\in L_{\theta}$, $-y \zeta = \overline{\zeta}$.
\end{lemma}
\begin{proof}
    This is straightforward by checking 
    \begin{equation*}
        -e^{i\theta} \cdot r e^{-i(\pi + \theta)/2} = - r e^{-i \pi/2 + i \theta/2} = r e^{i (\pi + \theta)/2}. 
    \end{equation*}
\end{proof}
The next lemma replaces the contour in Lemma~\ref{lem:contour} with $L_{\theta}$. The orientation of $L_{\theta}$ (see~Figure~\ref{fig:rotated-contour}) is chosen so that $L_{\theta}$ is the (counterclockwise) boundary of the half-plane containing positive knots.
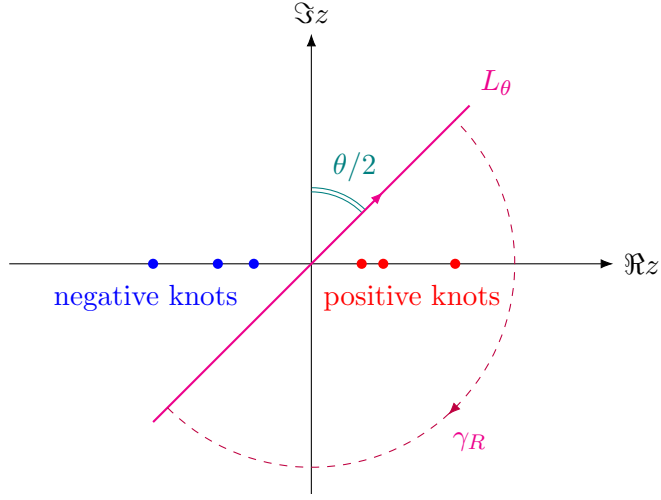
\begin{figure}[!htbp]
\centering
\begin{tikzpicture}[scale=0.95,>=Latex]
  \draw[->] (-4.2,0)--(4.2,0) node[right] {$\Re z$};
  \draw[->] (0,-3.2)--(0,3.2) node[above] {$\Im z$};
  \foreach \x in {-2.2,-1.3,-0.8} {\fill[blue] (\x,0) circle (2pt);}
  \foreach \x in {0.7,1.0,2.0} {\fill[red] (\x,0) circle (2pt);}
  \node[blue,below] at (-2.3,-0.15) {negative knots};
  \node[red,below] at (1.4,-0.15) {positive knots};
  \node[magenta,below] at (2.2,-2.2) {$\gamma_R$};
  \node[teal,below] at (0.6, 1.7) {$\theta/2$};
  \draw[teal] (0.75, 0.75) arc (45:90:1.05);
  \draw[teal] (0.707, 0.707) arc (45:90:1);
  \draw[thick,magenta] (-2.2,-2.2)--(2.2,2.2) node[above right] {$L_\theta$};
  \draw[magenta,->] (0,0) --  (1.0,1.0);

  \draw[purple, dashed] (-2,-2) arc (-135:45:2.828);
  \draw[purple,->] (2,-2)--(1.9, -2.1);
\end{tikzpicture}
\caption{Illustration of $L_{\theta}$.}
\label{fig:rotated-contour}
\end{figure}
\begin{lemma}\label{lem: U line int}
    For $|\theta| < \pi$, 
    \begin{equation}\label{eq: U con}
        U(e^{i\theta}) = -\frac{\eta_k - \eta_0}{2\pi i} \int_{L_{\theta}} \frac{\Omega_{\ast}(-e^{i\theta}\zeta)}{\Omega(\zeta)} d\zeta. 
    \end{equation}
    Equivalently, we can write the contour integral as
    \begin{equation}\label{eq: U mod}
    \begin{aligned}
            U(e^{i\theta}) &= -\frac{1}{2\pi i} \int_{L_{\theta}} \Phi_{\theta}(\zeta) \left(\frac{1}{\zeta - \eta_k} - \frac{1}{\zeta - \eta_0} \right) d\zeta \\
            &=\frac{-(\eta_k - \eta_0)}{2\pi i} \int_{L_{\theta}} \Phi_{\theta}(\zeta) \frac{1}{(\zeta - \eta_k)(\zeta - \eta_0)}   d\zeta,
    \end{aligned}
    \end{equation}
    where 
    \begin{equation*}
    \Phi_{\theta}(\zeta) = \prod_{j=1}^{k-1} \frac{-e^{i\theta}\zeta - \eta_j}{\zeta - \eta_j},    
    \end{equation*}
    and $ \left|\Phi_{\theta}(\zeta)\right| = 1$ for all $\zeta\in L_{\theta}.$
\end{lemma}

\begin{proof}
Let us first assume $\eta_j \neq 0$ for all interior knots. 
Take the contour path that follows the semicircle path enclosed by $\gamma_R$ shown in Figure~\ref{fig:rotated-contour} and $L_{\theta}$. The radius $R$ is sufficiently large to include all positive knots. Since the integrand decays like $O(|\zeta|^{-2})$ on $\gamma_R$, we obtain~\eqref{eq: U con} by taking $R\to\infty$. The identity~\eqref{eq: U mod} follows by a direct factorization and $|\Phi_{\theta}(\zeta)| = 1$ on $L_{\theta}$ holds by 
\begin{equation}\label{eq: modu eq}
    |-y\zeta - \eta_j| = |\overline{\zeta} - \eta_j| = |\zeta - \eta_j|,
\end{equation}
from Lemma~\ref{lem: rot conj} and $\eta_j\in\bbR$.

If one or more knots $\eta_j = 0$, $1\le j\le k-1$, then $\zeta = 0$ is a removable singularity. The same claim in Lemma~\ref{lem: U line int} still holds.
\end{proof}

\subsection{Geometric estimates}
The next lemma treats the binomial summation~\eqref{eq: bi-sum} as an integral instead. The key point is that if $|\theta|$ is far from zero, the factor $|1+e^{-i\theta}|^k \ll 2^k$ as $k\to\infty$, which implies we only have to focus on a narrow window near $\theta = 0$. This observation leads to the main difference from~\cite{SchererShadrin1999}. Using Theorem~\ref{thm: SS}, we have the following bound for $U(e^{i\theta})$.
\begin{lemma}\label{lem: easy bound}
$|U(e^{i\theta})|\le k$.
\end{lemma}
\begin{proof}
    Since each monomial term in $U(y)$ is bounded by one.
\end{proof}

\begin{lemma}\label{lem:binomial}
The generating polynomial $U(y)$ satisfies
\begin{equation}\label{eq:binomial-Fourier}
        \Sigma_k
        =\frac1{2\pi}\int_{-\pi}^{\pi}
        U(e^{i\theta})(1+e^{-i\theta})^k\dd\theta,
\end{equation}
where $\Sigma_k$ is the bionomial sum in~\eqref{eq: bi sum sigma}.
Moreover, for every $k\ge1$,
\begin{equation}\label{eq:binomial-L1-explicit}
        \frac1{2\pi}\int_{-\pi}^{\pi}|1+e^{-i\theta}|^k\dd\theta
        \le \frac{2^{k+1/2}}{\sqrt{\pi k}}.
\end{equation}
\end{lemma}

\begin{proof}
Expanding
$$
(1+e^{-i\theta})^k=\sum_{r=0}^{k}\binom{k}{r}e^{-ir\theta}
$$
and taking the constant Fourier coefficient leads to \eqref{eq:binomial-Fourier}.  Also,
$$
|1+e^{-i\theta}|^k=2^k|\cos(\theta/2)|^k.
$$
The Beta integral formula gives
$$
\frac1{2\pi}\int_{-\pi}^{\pi}|1+e^{-i\theta}|^k\dd\theta
=2^k\frac{\Gamma((k+1)/2)}{\sqrt\pi\,\Gamma(k/2+1)}.
$$
Then, the inequality~\eqref{eq:binomial-L1-explicit} is proved by using Gautschi's inequality:
$$
\frac{\Gamma((k+1)/2)}{\Gamma(k/2+1)}\le \sqrt{\frac{2}{k}}.
$$
\end{proof}
Now, we select a threshold $\theta_0 > 0$ which will be determined in Section~\ref{sec: proof}. The next lemma provides a tool that will later be used to obtain quantitative bounds for the contour integral.
\begin{lemma}\label{lem:distance}
For $|\theta|\le\theta_0$, $\zeta=\rho e^{-i(\pi+\theta)/2}\in L_{\theta}$, and any real $x$, 
\begin{equation}\label{eq:distance-explicit}
        |\zeta - x|\ge \delta_0^{1/2} (\rho^2+x^2)^{1/2}.
\end{equation}
where $\delta_0 := 1 - \sin(\theta_0/2)$.
\end{lemma}

\begin{proof}
For a real number $x$,
\begin{equation}\label{eq: cosine id}
     |\rho e^{-i(\pi+\theta)/2}-x|^2
 =\rho^2+x^2-2\rho x\cos\left(\frac{\pi+\theta}{2}\right).
\end{equation}
Since $\cos((\pi+\theta)/2)=-\sin(\theta/2)$ and $|\theta|\le\theta_0$,
$$
        \left|\cos\left(\frac{\pi+\theta}{2}\right)\right|
        \le \sin(\theta_0/2).
$$
Using $2|\rho x|\le \rho^2+x^2$, the right-side of~\eqref{eq: cosine id} is at least
$$
        (1-\sin(\theta_0/2))(\rho^2+x^2) = \delta_0 (\rho^2 + x^2).
$$
Thus $|\rho e^{-i(\pi+\theta)/2}-x|\ge \delta_0^{1/2}(\rho^2+x^2)^{1/2}$.
\end{proof}

\subsection{Phase extraction}
In this section, we provide an estimate of the product form as $\Phi_{\theta}$ as $|\zeta|$ grows large through phase extraction.
\begin{lemma}\label{lem: tri ineq}
Let $w_i = e^{i\alpha_i}$, $\alpha_i\in\bbR$, $i=1,2,3$, then 
\begin{equation*}
    |w_1 - 1| + |w_2 - 1| \ge |w_1 w_2 - 1|.
\end{equation*}
\end{lemma}
\begin{proof}
    This is exactly the triangle inequality on the complex plane.
\end{proof}

\begin{lemma}\label{lem:phase-extraction}
Let $\cI_{-}:= \{j \mid \eta_j < 0, \; 1\le j\le k-1\}$ be the index set of all negative interior knots, and set $$S_{-} = \sum_{j\in \cI_{-}} \eta_j.$$  For $|y|=1$, $|\theta| < \theta_0$, and $\zeta\in L_\theta$, define
$$
N_y(\zeta):=\prod_{j\in\cI_{-}}\frac{-y\zeta - \eta_j}{\zeta - \eta_j}.
$$
Then, 
\begin{equation}\label{eq:phase-extract-bound}
\begin{aligned}
|N_y(\zeta)-(-y)^{|\cI_{-}|}|
\le 2 \min\left\{1,\frac{|S_-|}{\delta_0^{1/2}|\zeta|}\right\}.
\end{aligned}
\end{equation}
\end{lemma}

\begin{proof}
Let
$$
G_{j}(\zeta) := (-y)^{-1} \frac{-y\zeta - \eta_j}{\zeta - \eta_j}
= \frac{\zeta + \eta_j y^{-1}}{\zeta - \eta_j} = 1 + \frac{(1 + y) \eta_j }{\zeta - \eta_j}.
$$
By~\eqref{eq: modu eq}, $|G_j(\zeta)| = 1$.
Thus, by Lemma~\ref{lem:distance},
$$
|G_j(\zeta) - 1| = \left| \frac{(1 + y) \eta_j }{\zeta - \eta_j} \right| \le \frac{2|\eta_j|}{\delta_0^{1/2}\sqrt{\zeta^2 + \eta_j^2}} < \frac{2 |\eta_j|}{\delta_0^{1/2} |\zeta|}.
$$
Using Lemma~\ref{lem: tri ineq} repeatedly on $\cI_{-}$,
\begin{equation*}
    \left|\prod_{j\in \Delta_{-}} G_j(\zeta) - 1\right| \le \sum_{j\in \cI_{-}} |G_j(\zeta) - 1| < \frac{2 |S_{-}|}{\delta_0^{1/2} |\zeta|}.
\end{equation*}
On the other hand, since both $N_y(\zeta)$ and $(-y)^{|\cI_{-}|}$ have modulus one, their difference is at most $2$. Hence, the bound of~\eqref{eq:phase-extract-bound} is proved.
\end{proof}
Following the same procedure in the proof of Lemma~\ref{lem:phase-extraction}, we have the following claim for positive interior knots. We omit the proof here.
\begin{corollary}\label{cor: pos}
Let $\cI_{+}:= \{j \mid \eta_j > 0, \; 1\le j\le k-1\}$ be the index set of all positive interior knots, and set $$S_{+} = \sum_{j\in \cI_{+}} \eta_j.$$  For $|y|=1$, $|\theta| < \theta_0$, and $\zeta\in L_\theta$, define
$$
P_y(\zeta):=\prod_{j\in\cI_{+}}\frac{-y\zeta - \eta_j}{\zeta - \eta_j}.
$$
Then
\begin{equation}\label{eq:phase-extract-bound-pos}
\begin{aligned}
|P_y(\zeta)-(-y)^{|\cI_{+}|}|
\le 2 \min\left\{1,\frac{|S_+|}{\delta_0^{1/2}|\zeta|}\right\}.
\end{aligned}
\end{equation}
\end{corollary}

\subsection{One-sided bound}
In this part, we consider an extreme case in which all \textit{interior} knots lie on one side of the origin (including the origin).
\begin{lemma}\label{lem: one-sided}
    If $\eta_j\ge 0$ for all $1\le  j\le k-1$, then 
    \begin{equation*}
        |U(y)|\le 1, \qquad |y| = 1.
    \end{equation*}
    The same claim holds if $\eta_j\le 0$ for all $1\le  j\le k-1$.
\end{lemma}
\begin{proof}
    Let the nodes be $\eta_0 < 0 \le \eta_1 \le \ldots \le \eta_{k-1} \le \eta_k$. We compute $U(y)$ through~\eqref{eq:U-def}, which is the product of $(\eta_0 - \eta_k)$ with the leading coefficient of the interpolation polynoimal for $\Omega_{\ast}(-yz)\one_{z > 0}$ at the knots. 

    Because the interpolation data $\Omega_{\ast}(-yz) 1_{z > 0}$ agrees with $\Omega_{\ast}(-yz)$ at all non-negative nodes, and vanishes at $\eta_0$, the interpolation polynomial can be computed by Newton's divided difference 
    \begin{equation*}
        Q_y(z) = \Omega_{\ast}(-yz) + c_y  \prod_{j=1}^k (z - \eta_j),
    \end{equation*}
    where the leading coefficient $c_y$ is derived using $Q_y(\eta_0) = 0$,
    \begin{equation*}
        c_y = - \frac{\Omega_{\ast}(-\eta_0 y)}{\prod_{j=1}^k (\eta_0 - \eta_j)}.
    \end{equation*}
    Hence, 
    \begin{equation*}
        U(y) = \prod_{j=1}^{k-1} \frac{-\eta_0 y - \eta_j}{\eta_0 - \eta_j } = \prod_{j=1}^{k-1} \frac{ \eta_j + \eta_0 y}{\eta_j - \eta_0 }.
    \end{equation*}
    Since each factor's modulus is bounded by one, $|U(y)|\le 1$. The same process applies when all interior knots are non-positive.
\end{proof}
\section{Main estimates}\label{sec: est}
In this section, we aim to bound $|U(e^{i\theta})|$ using~\eqref{eq: U mod} inside the window $|\theta| < \theta_0$. We will need the following key inequality.
\begin{lemma}
\label{lem:two-interval}
Assume $0 < A\le B$, $S\le kA$, and $k\ge2$.  Then
\begin{equation}
    \label{eq:two-interval-estimate}
    \begin{aligned}
 &(A+B)
 \int_0^\infty
 \min\left\{1,\frac{S}{\delta_0^{1/2}\rho}\right\}
 \frac{\dd\rho}{(\rho^2+A^2)^{1/2}(\rho^2+B^2)^{1/2}}  \\
 &\hspace{.7in}\le
2\left( \sinh^{-1}(1)+ 1
 + \log\frac{k}{\sqrt{\delta_0}}\right).    
    \end{aligned}
\end{equation}
The estimate also holds when $B\le A$, with $S\le kB$ by symmetry.
\end{lemma}

\begin{proof}
We split the integral at $\rho=A$.  On $[0, A]$, use the bound by one in the minimum and $(\rho^2+B^2)^{1/2}\ge B$.  Since $(A+B)/B\le2$, the contribution is at most
$$
       \frac{ (A+B) }{B}
        \int_0^A\frac{\dd\rho}{(\rho^2+A^2)^{1/2}}
        \le 2\sinh^{-1}(1) = 2\log(1+\sqrt{2}).
$$
On $[A,\infty)$, use $(\rho^2+A^2)^{1/2}\ge\rho$ and $(\rho^2+B^2)^{1/2}\ge B$.  With $R=S/\sqrt{\delta_0}$, the contribution is at most
$$
        \frac{A+B}{B}
        \int_A^\infty \min\left\{1,\frac{R}{\rho}\right\}\frac{\dd\rho}{\rho} \le 2 (1 + \log_{+}(R/A)).
$$
where $\log_{+}(x) = \max(0, \log(x))$.  Since $S\le kA$,
$$
        1+\log_+(R/A)
        \le 1+\log\frac{k}{\sqrt{\delta_0}}.
$$
Combining the two intervals gives \eqref{eq:two-interval-estimate}.
\end{proof}

\begin{proposition}\label{prop:saddle-bound}
For every $k\ge2$ and every $|\theta| < \theta_0$, we have
\begin{equation}\label{eq:saddle-bound-explicit}
        |U(e^{i\theta})|
        \le \alpha \log k + \beta,
\end{equation}
where $\delta_0 = 1 - \sin(\theta_0/2)$, and
\begin{equation*}
    \alpha = \frac{4}{\pi\delta_0},\qquad \beta = 1 + \frac{4\log(1+\sqrt{2}) + 4}{\pi \delta_0} - \frac{2\log\delta_0}{\pi\delta_0}.
\end{equation*}
\end{proposition}

\begin{proof}
Let $A = -\eta_0 > 0$ and $B = \eta_k > 0$. Assume first $A\le B$.  For $\zeta \in L_\theta$, write
$$
        \frac{\Omst(-y\zeta)}{\Omst(\zeta)}= N_y(\zeta)P_y(\zeta) (-y)^Z,
$$
where $N_y$ is the product over negative interior knots $\cI_{-}$ (see Lemma~\ref{lem:phase-extraction}), $P_y$ is the product over positive interior knots $\cI_{+}$ (see Corollary~\ref{cor: pos}), and $Z$ is the number of knots at the origin. Lemma~\ref{lem:phase-extraction} gives
$$
        N_y=(-y)^{|\cI_{-}|}+E_y,
        \qquad
        |E_y(\zeta)|\le 2\min\left\{1,\frac{|S_-|}{|\zeta|}\right\}.
$$
We obtain
\begin{equation}\label{eq: two term}
    \frac{\Omst(-y\zeta)}{\Omst(\zeta)} = P_y(\zeta) (-y)^{|\cI_{-}|} (-y)^Z + P_y(\zeta) E_y(\zeta) (-y)^Z.  
\end{equation}
Then, we substitute into \eqref{eq: U mod}.
The contribution of the first term from~\eqref{eq: two term} is the one-sided configuration with no negative interior nodes.  By Lemma~\ref{lem: one-sided}, this term's contribution in~\eqref{eq: U mod} has modulus at most $1$.  Since $|P_y|=1$ on $L_\theta$ and $|y| = 1$. Using Lemma~\ref{lem:distance}, the second term from~\eqref{eq: two term} gives an error estimate of at most
$$
2\cdot \frac{A + B}{2\pi \delta_0}\int_0^\infty
 \min\left\{2,\frac{2|S_-|}{\rho}\right\}
 \frac{\dd\rho}{(\rho^2+A^2)^{1/2}(\rho^2+B^2)^{1/2}}.
$$
The factor of $2$ in front is because $L_{\theta}$ is split into two half-lines. Since all negative interior knots $|\eta_j|\le A$ and there are at most $k-1$ such nodes, thus
$$
        S_-\le kA.
$$
By Lemma~\ref{lem:two-interval},  the error is bounded by
$$
\frac{4}{\pi\delta_0}\left( \sinh^{-1}(1)+ 1
 + \log\frac{k}{\sqrt{\delta_0}}\right). 
$$
If $B\le A$, the same argument is applied with $S_{+} \le k B$. Therefore, the bound~\eqref{eq:saddle-bound-explicit} is proved inside the window.
\end{proof}
This improves the earlier bound in Lemma~\ref{lem: easy bound} from $|U(e^{i\theta})| \le k$ to $O(\log k)$ scale for $|\theta|\le \theta_0$, which can significantly improve the estimate of $\Sigma_k$ using~\eqref{eq:binomial-Fourier}. 

\section{Proof of the main theorem}\label{sec: proof}

\begin{proof}[Proof of Theorem~\ref{thm:main}]
Using Proposition~\ref{prop:saddle-bound} and Lemma~\ref{lem: easy bound} on the central window $|\theta|\le \theta_0$ and its complement, respectively. We obtain

$$
 \frac1{2\pi}\int_{-\pi}^{\pi}
 |U(e^{i\theta})|\,|1+e^{-i\theta}|^k\dd\theta
 \le B_k (\alpha \log k + \beta) + k(A_k - B_k),
$$
where $\alpha, \beta$ are from Proposition~\ref{prop:saddle-bound}, and $$A_k = \frac{1}{2\pi}\int_{-\pi}^{\pi} |1 + e^{-i\theta}|^k d\theta = 2^k\frac{\Gamma((k+1)/2)}{\sqrt\pi\,\Gamma(k/2+1)},\qquad B_k = \frac{1}{2\pi}\int_{|\theta|\le \theta_0} |1 + e^{-i\theta}|^k d\theta.$$
We take $\theta_0 = \frac{\pi}{3}$, $\delta_0 = \frac{1}{2}$. Then $\alpha = \frac{8}{\pi} < 2.55$, $\beta = 1 + \frac{4\log(2)}{\pi} + \frac{8 + 8 \log(1+\sqrt{2})}{\pi} < 6.7$, and
\begin{equation*}
\begin{aligned}
    A_k - B_k &= \frac{2^{k+1}}{\pi} \int_{\theta_0}^{\pi/2} \cos^{k}(t) dt \\&< \frac{2^{k+1}}{\pi \sin(\theta_0/2)} \int_{\theta_0/2}^{\pi/2} \cos^{k}(t) \sin(t) dt = \frac{2^{k+1}}{\pi (k+1) \sin(\theta_0/2)} \cos^{k+1}(\theta_0/2)\\
    & =  2^{k} \frac{4}{\pi (k+1) } \left(\frac{\sqrt{3}}{2}\right)^{k+1}.
\end{aligned}
\end{equation*}
And $\frac{4}{\pi}\left(\frac{\sqrt{3}}{2}\right)^{k+1}\sqrt{k} < \frac{8}{\pi} < 2.55$ for all $k\ge 2$.
 Therefore, using $B_k < A_k$, we obtain
\begin{equation*}
\begin{aligned}
    \Sigma_k &< 2^{k+1/2}\frac{1}{\sqrt{\pi k}} (2.55 \log k + 6.7) + 2^k k^{-1/2} \cdot 2.55\\
    &< 3 \cdot 2^k  k^{-1/2} (\log k + 3).
\end{aligned}
\end{equation*}
Hence, by Theorem~\ref{thm: SS}, we obtain the upper bound 
\begin{equation*}
    \kappa_{k, p}\le 3 \cdot 2^k  k^{1/2} (\log k + 3),
\end{equation*}
which improves from $k 2^k$ to $O(\sqrt{k}\log k\,2^k)$. 

Moreover, if we choose $\theta_0 = O(k^{-1/2}\sqrt{\log k})$ as $k\to\infty$, we will get $\delta_0 \to 1$ and the contribution outside the window is still  negligible. Then the constant in front of $k^{1/2}\log k\,2^k$ can be optimzed to $\frac{4\sqrt{2}}{\pi^{3/2}} \approx 1.016$.
\end{proof}

\section{Concluding Remarks}\label{sec: con}
This note improves the previous upper bound of the B-spline condition number. The main technique is to treat the binomial sum~\eqref{eq: bi sum sigma} as an integral over the generating polynomial $U$ and reduce the problem to estimating $|U|$. However, as outlined in~\cite{SchererShadrin1999}, this approach cannot overcome the final $\sqrt{k}$ factor to attain a proof for DeBoor's conjecture since the binomial term $\binom{k}{\lfloor k/2 \rfloor}$ is already at $O(2^k/\sqrt{k})$. However, it is possible to further reduce the logarithmic factor $\log k$ from Theorem~\ref{thm:main}. When considering the contribution of the error term in~\eqref{eq: two term}, the possibility of phase cancellation was overlooked. 

\section*{Acknowledgments}
The author would like to thank Kui Ren for suggestions and comments. The author was partially supported by NSF grant DMS-2309530. 

\bibliographystyle{siam}
\bibliography{main}

\end{document}

%% file: preamble.tex
\pdfoutput=1
\usepackage{graphicx}
\graphicspath{{./Figures/}}
\usepackage{subfigure}

\usepackage{amssymb,latexsym}

\usepackage{setspace}

\usepackage{amsmath}

\usepackage{mathrsfs,amsfonts}

\usepackage{amsthm,amsxtra}

\usepackage[final]{showkeys}

\usepackage{stmaryrd}

\usepackage{algorithm}
\usepackage{algorithmicx}
\usepackage{algpseudocode}
\usepackage{mathtools}
\newif\ifPDF
\ifx\pdfoutput\undefined
\PDFfalse
\else
\ifnum\pdfoutput > 0
\PDFtrue
\else
\PDFfalse
\fi
\fi

\ifPDF
\usepackage{pdftricks}
\begin{psinputs}
	\usepackage{pstricks}
	\usepackage{pstcol}
	\usepackage{pst-plot}
	\usepackage{pst-tree}
	\usepackage{pst-eps}
	\usepackage{multido}
	\usepackage{pst-node}
	\usepackage{pst-eps}
\end{psinputs}
\else
\usepackage{pstricks}
\fi

\ifPDF
\usepackage[debug,pdftex,colorlinks=true, 
linkcolor=blue, bookmarksopen=false,
plainpages=false,pdfpagelabels]{hyperref}
\else
\usepackage[dvips]{hyperref}
\fi

\usepackage[openbib]{currvita}

\usepackage{fancyhdr}

\newtheorem{theorem}{Theorem}[section]
\newtheorem{lemma}[theorem]{Lemma}

\newtheorem{proposition}[theorem]{Proposition} 
 
\newtheorem{corollary}[theorem]{Corollary}

\newcommand{\bbC}{\mathbb C}

\newcommand{\bbR}{\mathbb R}

\newcommand{\cI}{\mathcal I}

\newenvironment{keywords}
{\noindent{\bf Key words.}\small}{\par\vspace{1ex}}
\newenvironment{AMS}
{\noindent{\bf AMS subject classifications 2000.}\small}{\par}

\makeatletter
\newcommand{\chapterauthor}[1]{%
	{\parindent0pt\vspace*{-25pt}%
		\linespread{1.1}\large\scshape#1%
		\par\nobreak\vspace*{35pt}}
	\@afterheading%
}
\makeatother